\documentclass[11pt,twoside]{article}
\usepackage{amsmath}
\usepackage{amsfonts}
\usepackage{amsthm,amssymb}
\usepackage{graphicx, color}

\usepackage{indentfirst}

\usepackage[nosort]{cite}
\usepackage{color}
\usepackage{xcolor}%
\usepackage{makeidx}

\newtheorem{theorem}{Theorem}

\newtheorem{lemma}[theorem]{Lemma}
\newtheorem{proposition}[theorem]{Proposition}

\newtheorem{corollary}[theorem]{Corollary}

\newtheorem{claim}{Claim}

\begin{document}
\vskip 1cm
\begin{center}
\vskip 0.1cm {\bf\Large On existence and concentration of solutions to a class of quasilinear problems involving the $1-$Laplace operator}
\end{center}
\vskip 0.3cm
\begin{center}

{\sc Claudianor O. Alves$^1$ and Marcos T. O. Pimenta$^{2,*}$,}
\\
\vspace{0.5cm}

1. Unidade Acad\^emica de Matem\'atica \\ Universidade Federal de Campina Grande \\
58429-900 - Campina Grande - PB , Brazil \\

2. Departamento de Matem\'atica e Computa\c{c}\~ao\\ Universidade Estadual Paulista (Unesp), Faculdade de Ci\^encias e Tecnologia\\
19060-900 - Presidente Prudente - SP, Brazil, \\
* corresponding author
\medskip

E-mail addresses: coalves@dme.ufcg.edu.br, pimenta@fct.unesp.br

 \end{center}
\vskip 1cm

\begin{abstract}
In this work we use variational methods to prove results on existence and concentration of solutions to a problem in $\mathbb{R}^N$ involving the $1-$Laplacian operator. A thorough analysis on the energy functional defined in the space of functions of bounded variation $BV(\mathbb{R}^N)$ is necessary, where the lack of compactness is overcome by using the Concentration of Compactness Principle of Lions.

\end{abstract}

\vskip 1.5cm

 \vskip 0.4cm

\section{Introduction and some abstract results}
Several recent studies have focused on the nonlinear Schr\"{o}dinger
equation
$$
i\epsilon\displaystyle \frac{\partial \Psi}{\partial t}=-\epsilon^{2}\Delta
\Psi+(V(z)+E)\Psi-f(\Psi)\,\,\, \mbox{for all}\,\,\, z \in
\mathbb{R}^{N},\eqno{(NLS)}
$$
where $N \geq 2$, $\epsilon > 0$ is a positive parameter and $V,f$ are continuous function verifying some conditions. This class of equation is one of the main objects of the quantum physics, because it appears in problems involving nonlinear optics, plasma physics and condensed matter physics.

The knowledge of the solutions for the elliptic equation
$$
\ \  \left\{
\begin{array}{l}
- \epsilon^{2} \Delta{u} + V(z)u=f(u)
\ \ \mbox{in} \ \ \mathbb{R}^{N},
\\
u \in H^{1}(\mathbb{R}^{N}),
\end{array}
\right.
\eqno{(S)_{\epsilon}}
$$
%or equivalently for the elliptic equation
%$$
%\ \  \left\{
%\begin{array}{l}
%-\Delta{u} + V(\epsilon z)u=f(u)
%\ \ \mbox{in} \ \ \mathbb{R}^{N},
%\\
%u \in H^{1}(\mathbb{R}^{N})
%\end{array}
%\right.
%\eqno{(S')_{\epsilon}},
%$$
has a great importance in the study of standing-wave solutions
of $(NLS)$. The existence and concentration of
positive solutions for general semilinear elliptic equations
$(S)_\epsilon$ for the case $N \geq 3$ have been extensively
studied, see for example, Floer and Weinstein \cite{FW}, Oh
\cite{O2}, Rabinowitz \cite{rabinowitz}, Wang \cite{WX}, Cingolani and Lazzo \cite{CL97},
Ambrosetti, Badiale and Cingolani  \cite{ABC}, Gui \cite{G}, del Pino and Felmer \cite{DF1}  and their references.

In the above mentioned papers, the existence,  multiplicity and concentration of positive solutions have been obtained in connection with the geometry of the function $V$. In \cite{rabinowitz}, by a mountain pass argument,  Rabinowitz proves the existence of positive solutions of $(S)_{\epsilon}$ for $\epsilon > 0$  
small and
$$
\liminf_{|z| \rightarrow \infty} V(z) > \inf_{z \in
	\mathbb{R}^N}V(z)=V_{0} >0.
$$
Later Wang \cite{WX} showed that these solutions concentrate at global minimum points of $V$ as  $\epsilon$ tends to 0. In \cite{DF1}, del Pino and Felmer have found solutions which concentrate around local minimum of $V$ by introducing a penalization method. More precisely, they assume that there is an open and bounded set $\Lambda \subset \mathbb{R}^N$ such that
$$
0< V_{0}\leq \inf_{z\in\Lambda}V(z)< \min_{z \in
	\partial\Lambda}V(z).
$$

Motivated by papers \cite{rabinowitz} and \cite{WX}, let us consider the following class of quasilinear elliptic problems
\begin{equation}
\left\{
\begin{array}{rr}
\displaystyle - \epsilon \Delta_1 u + V(x)\frac{u}{|u|} & = f(u) \quad \mbox{in $\mathbb{R}^N$,}\\
& u \in BV(\mathbb{R}^N),
\end{array} \right.
\label{Pintro}
\end{equation}
where $\epsilon > 0$, $N \geq 2$ and the operator $\Delta_1$ is the well known $1-$Laplacian operator, whose formal definition is given by $\displaystyle \Delta_1 u = \mbox{div}\left(\frac{\nabla u}{|\nabla u|}\right)$. The nonlinearity $f$ is assumed to satisfy the following set of assumptions:
\begin{itemize}
\item [$(f_1)$] $f \in C^1(\mathbb{R})$;
\item [$(f_2)$] $f(s) = o(1)$ as $s \to 0$;
\item [$(f_3)$] There exist constants $c_1, c_2 > 0$ and $p \in [1,1^*)$ such that
$$
|f(s)| \leq c_1 + c_2|s|^{p-1} \quad \forall s \in \mathbb{R}; 
$$
\item [$(f_4)$] There exists $\theta > 1 $ such that $$0 < \theta F(s) \leq f(s)s, \quad \mbox{for $s \neq 0$},$$
where $ \displaystyle F(s) = \int_0^s f(t)dt$;
\item [$(f_5)$] $f$ is increasing.
\end{itemize}
Hereafter, the potential is going to be considered satisfying some of the following conditions:
\begin{itemize}
\item [$(V_1)$] $V \in L^\infty(\mathbb{R}^N)$ and $\displaystyle 0 < V_0:=\inf_{\mathbb{R}^N}V$;
\item [$(V_2)$] $\displaystyle  V_\infty:=\liminf_{|x| \to +\infty}V(x) > V_0$;
\item [$(V_3)$] $V \in C(\mathbb{R}^N)$, $\displaystyle  \lim_{ |x| \to +\infty} V(x) = V_\infty$ and $V \leq V_\infty$ in $\mathbb{R}^N$.
\end{itemize}

Hereafter, we will say that $V$ satisfies the Rabinowitz's condition when $(V_1)-(V_2)$ hold.  

By studying problem (\ref{Pintro}) we are looking to get some results on existence and concentration of solutions, as the parameter $\epsilon \to 0^+$. The approach used as in the laplacian case is variational. However, the right space in which problem (\ref{Pintro}) takes place is the space of functions of bounded variation, $BV(\mathbb{R}^N)$. The energy function associated to (\ref{Pintro}) is $I_\epsilon: BV(\mathbb{R}^N) \to \mathbb{R}$, defined by
$$
I_\epsilon(u) = \epsilon\int_{\mathbb{R}^N}|Du| + \int_{\mathbb{R}^N}V(x)|u|dx - \int_{\mathbb{R}^N}F(u)dx,
$$
where $|Du|$ is the total variation of the vectorial Radon measure $Du$ (see Section 2).

Before we state our main results, we would like to mention the main difficulties in dealing with (\ref{Pintro}), which are organized in the list bellow:
\begin{itemize}
\item Problem (\ref{Pintro}) is just a formal version of the correct Euler-Lagrange equation associated to the functional $I_\epsilon$, since it is not well defined wherever $\nabla u$ or $u$ vanishes. The correct one, i.e., the equation satisfied by the critical points of $I_\epsilon$ is given by
$$
\left\{
\begin{array}{l}
\exists z \in L^\infty(\mathbb{R}^N,\mathbb{R}^N), \, \, |z|_\infty \leq 1,\, \,  \mbox{div}z \in L^N(\mathbb{R}^N), \, \, -\int_{\mathbb{R}^N}u \mbox{div}z dx = \int_{\mathbb{R}^N}|Du|,\\
\exists z_2^* \in L^N(\mathbb{R}^N),\, \, z_2^*V(x)|u| = u \quad \mbox{a.e. in $\mathbb{R}^N$},\\
-\epsilon \mbox{div} z + z_2^* = f(u), \quad \mbox{a.e. in $\mathbb{R}^N$},
\end{array}
\right.
$$
and is going to be obtained in Section 2.1; \item The functional $I_\epsilon$ is not $C^1(BV(\mathbb{R}^N))$ and then some other sense of critical point have to be considered. Since $I_\epsilon$ is written like the difference between a convex locally Lipschitz functional and a smooth one, the theory of sub-differential of Clarke (see \cite{Clarke,Chang}) can be applied. Following this theory, it is possible to define a sense of critical point, Palais-Smale sequence, etc., that provide us with the tools to carry a variational approach to (\ref{Pintro});
\item The space $BV(\mathbb{R}^N)$, the domain of $I_\epsilon$, is not reflexive neither uniformly convex. This is the reason why is so difficult to prove that the functionals defined in this space satisfy compactness conditions like the Palais-Smale one;
\item The solutions we will get lacks smoothness, then a lot of arguments explored in the literature cannot be used here, like convergence in the sense $C_{loc}^{2}(\mathbb{R}^N)$, $C_{loc}^{1}(\mathbb{R}^N)$, etc.

\item To overcome the above difficulties we have used in Section 3 the Concentration of Compactness Principle due to Lions, which is in our opinion an important novelty in the study of concentration of solution. Here, we must observe that our approach  can also be used for concentration problem involving the laplacian operator.  

\item Here to get a solution we must prove that if $(v_n)$ is a Palais-Smale sequence associated with the energy functional we must have
$$
\int_{\mathbb{R}^N}f(v_n)v_n dx \to \int_{\mathbb{R}^N}f(v)v dx.
$$ 
For a lot of problems involving the Laplacian the above limit is not necessary to get a nontrivial solution, however for our problem this limit is crucial.

\end{itemize}

Our main results are the following theorems.

\begin{theorem}
Suppose that $f$ satisfies  the conditions $(f_1) - (f_5)$ and that $V$ satisfies $(V_1)$ and $(V_2)$. Then there exist $\epsilon_0 > 0$ such that (\ref{Pintro}) has a nontrivial bounded variation solution $u_\epsilon$, for all $0 < \epsilon < \epsilon_0$. Moreover, for each sequence $\epsilon_n \to 0$, up to a subsequence, the family $(u_{\epsilon_n})_{n \in \mathbb{N}}$ concentrate around a point $x_0 \in \mathbb{R}^N$ such that $V(x_0) = V_0$. More specifically, there exists $C > 0$ such that for all $\delta > 0$, there exist $\overline{R} > 0$ and $n_0 \in \mathbb{N}$ such that, \begin{equation}
\int_{B^c_{\epsilon_n\overline{R}}(x_0)}f(u_n)u_ndx < \epsilon_n^N\delta \quad \mbox{and} \quad \int_{B_{\epsilon_n\overline{R}}(x_0)}f(u_n)u_ndx \geq C\epsilon_n^N,
\label{concentrationf}
\end{equation}
for all $n \geq n_0$.
\label{theoremapplication1}
\end{theorem}

Our second result shows the existence of solution for all $\epsilon >0$ when $V$ is asymptotically linear and it has the following statement.

\begin{theorem} 
Suppose that $f$ satisfies  the conditions $(f_1) - (f_5)$ and also $(V_1)$ and $(V_3)$, then there exist a nontrivial bounded variation solution $u_\epsilon$ of (\ref{Pintro}) for all $\epsilon > 0$.
\label{theoremapplication2}
\end{theorem}

Here, we would like point out that a version of Theorem \ref{theoremapplication2} for Laplacian operator was proved by Jianfu and Xiping \cite{JX}. 

Before concluding this section, we would like to mention some paper involving the $\Delta_1$ on bounded domain, where the reader can find more references about this subject. In \cite{DegiovanniMagrone}, Degiovanni and Magrone study the version of Br\'ezis-Nirenberg problem to the 1-Laplacian operator, corresponding to 
$$
\left\{
\begin{array}{rl}
\displaystyle - \Delta_1 u & = \displaystyle  \lambda \frac{u}{|u|} + |u|^{1^*-2}u \quad \mbox{in
$\Omega$,}\\
u & = 0 \quad \mbox{on
$\partial\Omega$.}\\
\end{array} \right.
$$
In \cite{Chang2}, Chang uses this approach to study the spectrum of the $1-$Laplacian operator, proving the existence of a sequence of eigenvalues. In \cite{Kawohl}, Kawohl and Schuricht also study the spectrum of the $1-$Laplacian operator and reach the astonishing conclusion that an eigenfunction of this operator, in general satisfies infinity many Euler-Lagrange equations associated with it.

The paper is organized as follows. In Section 2 we give a brief overview about the space $BV(\mathbb{R}^N)$, define the sense of solution we are going to deal with and also find the precise Euler-Lagrange equation associated to $I_\epsilon$. In Section 3 we prove Theorem \ref{theoremapplication1}, studying separately the arguments on existence and concentration of the solutions. In Section 4 we prove Theorem \ref{theoremapplication2}. Finally, in the last section we prove the existence of a ground-state solution to the autonomous problem.

\section{Preliminary results}

First of all let us note that the problem (\ref{Pintro}), through the change of variable $v(x) = u(\epsilon x)$, is equivalent to the problem
\begin{equation}
\left\{
\begin{array}{rr}
\displaystyle - \Delta_1 v + V(\epsilon x)\frac{v}{|v|} & = f(v) \quad \mbox{in $\mathbb{R}^N$,}\\
& u \in BV(\mathbb{R}^N),
\end{array} \right.
\label{Pintrov}
\end{equation}

Let us introduce the space of functions of bounded variation, $BV(\mathbb{R}^N)$. We say that $u \in BV(\mathbb{R}^N)$, or is a function of bounded variation, if $u \in L^1(\mathbb{R}^N)$, and its distributional derivative $Du$ is a vectorial Radon measure, i.e., 
$$BV(\mathbb{R}^N) = \left\{u \in L^1(\mathbb{R}^N); \, Du \in \mathcal{M}(\mathbb{R}^N,\mathbb{R}^N)\right\}.$$
It can be proved that $u \in BV(\mathbb{R}^N)$ is equivalent to $u \in L^1(\mathbb{R}^N)$ and
$$\int_{\mathbb{R}^N} |Du| := \sup\left\{\int_{\mathbb{R}^N} u \mbox{div}\phi dx; \, \, \phi \in C^1_c(\mathbb{R}^N,\mathbb{R}^N), \, \mbox{s.t.} \, \, |\phi|_\infty \leq 1\right\} < +\infty.$$

The space $BV(\mathbb{R}^N)$ is a Banach space when endowed with the norm
$$\|u\| := \int_{\mathbb{R}^N} |Du| + |u|_1,$$
which is continuously embedded into $L^r(\mathbb{R}^N)$ for all $\displaystyle r \in \left[1,1^*\right]$.

As one can see in \cite{Buttazzo}, the space $BV(\mathbb{R}^N)$ has different convergence and density properties than the usual Sobolev spaces. For example, $C^\infty_0(\mathbb{R}^N)$ is not dense in $BV(\mathbb{R}^N)$ with respect to the strong convergence, since $\overline{C^\infty_0(\mathbb{R}^N)}$ w.r.t. the  $BV(\mathbb{R}^N)$ norm is equal to $W^{1,1}(\mathbb{R}^N)$, a proper subspace of $BV(\mathbb{R}^N)$. This has motivated people to define a weaker sense of convergence in $BV(\mathbb{R}^N)$, called {\it intermediate convergence}. We say that $(u_n) \subset BV(\mathbb{R}^N)$ converge to $u \in BV(\mathbb{R}^N)$ in the sense of the intermediate convergence if 
$$
u_n \to u, \quad \mbox{in $L^1(\mathbb{R}^N)$}
$$
and
$$
\int_{\mathbb{R}^N}|Du_n| \to \int_{\mathbb{R}^N}|Du|,
$$
as $n \to \infty$. Fortunately, with respect to the intermediate convergente, $C^\infty_0(\mathbb{R}^N)$ is dense in $BV(\mathbb{R}^N)$.

For a vectorial Radon measure $\mu \in \mathcal{M}(\mathbb{R}^N,\mathbb{R}^N)$, we denote by $\mu = \mu^a + \mu^s$ the usual decomposition stated in the Radon Nikodyn Theorem, where $\mu^a$ and $\mu^s$ are, respectively, the absolute continuous and the singular parts with respect to the $N-$dimensional Lebesgue measure $\mathcal{L}^N$. We denote by $|\mu|$, the absolute value of $\mu$, the scalar Radon measure defined like in \cite{Buttazzo}[pg. 125]. By $\displaystyle \frac{\mu}{|\mu|}(x)$ we denote the usual Lebesgue derivative of $\mu$ with respect to $|\mu|$, given by
$$\frac{\mu}{|\mu|}(x) = \lim_{r \to 0}\frac{\mu(B_r(x))}{|\mu|(B_r(x))}.$$ 

It can be proved that $\mathcal{J}: BV(\mathbb{R}^N) \to \mathbb{R}$, given by
\begin{equation}
\mathcal{J}(u) = \int_{\mathbb{R}^N} |Du| + \int_{\mathbb{R}^N} |u|dx,
\label{J}
\end{equation}
is a convex functional and Lipschitz continuous in its domain. It is also well know that $\mathcal{J}$ is lower semicontinuous with respect to the $L^r(\mathbb{R}^N)$ topology, for $r \in [1,1^*]$ (see \cite{Giusti} for example). Although non-smooth, the functional $\mathcal{J}$ admits some directional derivatives. More specifically, as is shown in \cite{Anzellotti}, given $u \in BV(\mathbb{R}^N)$, for all $v \in BV(\mathbb{R}^N)$ such that $(Dv)^s$ is absolutely continuous w.r.t. $(Du)^s$ and such that $v$ is equal to $0$ a.e. in the set where $u$ vanishes, it follows that
\begin{equation}
\mathcal{J}'(u)v = \int_{\mathbb{R}^N} \frac{(Du)^a(Dv)^a}{|(Du)^a|}dx + \int_{\mathbb{R}^N} \frac{Du}{|Du|}(x)\frac{Dv}{|Dv|}(x)|(Dv)|^s + \int_{\mathbb{R}^N}\mbox{sgn}(u) v dx,
\label{Jlinha}
\end{equation}
where $\mbox{sgn}(u(x)) = 0$ if $u(x) = 0$ and $\mbox{sgn}(u(x)) = u(x)/|u(x)|$ if $u(x) \neq 0$.
In particular, note that, for all $u \in BV(\mathbb{R}^N)$,
\begin{equation}
\mathcal{J}'(u)u = \mathcal{J}(u).
\label{derivadaJ}
\end{equation}

Let us define in the space $BV(\mathbb{R}^N)$ the following norms,
$$
\|v\|_\epsilon := \int_{\mathbb{R}^N}|Dv| + \int_{\mathbb{R}^N}V(\epsilon x)|v|dx,
$$
$$
\|v\|_\infty := \int_{\mathbb{R}^N}|Dv| + \int_{\mathbb{R}^N}V_\infty|v|dx,
$$
and
$$
\|v\|_0 := \int_{\mathbb{R}^N}|Dv| + \int_{\mathbb{R}^N}V_0|v|dx,
$$
which by $(V_1)$ and, $(V_2)$ or $(V_3)$, are equivalent to the usual norm in $BV(\mathbb{R}^N)$.

Let us define also the functionals $\Phi_\epsilon, \Phi_\infty, \Phi_0: BV(\mathbb{R}^N) \to \mathbb{R}$ by
$$
\Phi_\epsilon(v) = \|v\|_\epsilon - \int_{\mathbb{R}^N}F(v)dx,
$$
$$
\Phi_\infty(v) = \|v\|_\infty - \int_{\mathbb{R}^N}F(v)dx
$$
and
$$
\Phi_0(v) = \|v\|_0 - \int_{\mathbb{R}^N}F(v)dx.
$$

Denoting $\mathcal{F}(v) = \int_{\mathbb{R}^N}F(v)dx$ and $\mathcal{J}_\epsilon(v) = \|v\|_\epsilon$, note that $\mathcal{F} \in C^1(BV(\mathbb{R}^N))$ and $\mathcal{J}_\epsilon$ defines a locally Lipschitz continuous functional. Then we say that $v_\epsilon \in BV(\mathbb{R}^N)$ is a solution of (\ref{Pintrov}) if $0 \in \partial \Phi_\epsilon(v_\epsilon)$, where $\partial \Phi_\epsilon(v_\epsilon)$ denotes the generalized gradient of $\Phi_\epsilon$ in $v_\epsilon$, as defined in \cite{Chang}. It follows that this is equivalent to $\mathcal{F}'(v_\epsilon) \in \partial \mathcal{J}_\epsilon(v_\epsilon)$ and, since $\mathcal{J}_\epsilon$ is convex, this is written as
\begin{equation}
\|w\|_\epsilon - \|v_\epsilon\|_\epsilon \geq \int_{\mathbb{R}^N}f(v_\epsilon)(w - v_\epsilon)dx, \quad \forall w\in BV(\mathbb{R}^N).
\label{eqsolution}
\end{equation}
Hence all $v_\epsilon \in BV(\mathbb{R}^N)$ such that (\ref{eqsolution}) holds is going to be called a bounded variation solution of (\ref{Pintrov}). Analogously we define critical points of the functionals $\Phi_\infty$ and $\Phi_0$, since they have the same properties that $\Phi_\epsilon$.

\subsection{The Euler-Lagrange equation}

Since (\ref{Pintrov}) contains expressions that doesn't make sense when $\nabla u = 0$ or $u = 0$, then it can be understood just as the formal version of the Euler-Lagrange equation associated to the functional $\Phi_\epsilon$. In this section we present the precise form of an Euler-Lagrange equation satisfied by all bounded variation critical points of $\Phi_\epsilon$. In order to do so we closely follow the arguments in \cite{Kawohl}.

The first step is to consider the extension of the functionals $\mathcal{J}_\epsilon, \mathcal{F}$ and $\Phi_\epsilon$ to $L^{1^*}(\mathbb{R}^N)$, given respectively by $\overline{\mathcal{J}}_\epsilon, \overline{\mathcal{F}}, \overline{\Phi}_\epsilon: L^{1^*}(\mathbb{R}^N) \to \mathbb{R}$, where
$$
\overline{\mathcal{J}}_\epsilon(v) = 
\left\{
\begin{array}{ll}
\mathcal{J}_\epsilon(v), & \mbox{if $v \in BV(\mathbb{R}^N)$},\\
+\infty, & \mbox{if $v \in L^{1^*}(\mathbb{R}^N)\backslash BV(\mathbb{R}^N)$},
\end{array}
\right.
$$
$$
\overline{\mathcal{F}}(u) = \int_{\mathbb{R}^N}F(u)dx
$$
and $\overline{\Phi}_\epsilon = \overline{\mathcal{J}}_\epsilon - \overline{\mathcal{F}}$. It is easy to see that $\overline{\mathcal{F}}$ belongs to $C^1(L^{1^*}(\mathbb{R}^N), \mathbb{R})$ and that $\overline{\mathcal{J}}_\epsilon$ is a convex lower semicontinuous functional defined in $L^{1^*}(\mathbb{R}^N)$. Hence the subdifferential (in the sense of \cite{Szulkin}) of $\overline{\mathcal{J}}_\epsilon$, denoted by $\partial \overline{\mathcal{J}}_\epsilon$, is well defined. The following is a crucial result in obtaining an Euler-Lagrange equation satisfied by the critical points of $\Phi_\epsilon$.

\begin{lemma}
If  $v_\epsilon \in BV(\mathbb{R}^N)$ is such that $0 \in \partial \Phi_\epsilon(v_\epsilon)$, then $0 \in \partial \overline{\Phi_\epsilon}(v_\epsilon)$.
\end{lemma}
\begin{proof}
Suppose that $0 \in \partial \Phi_\epsilon(v_\epsilon)$, i.e., that $v_\epsilon$ satisfies (\ref{eqsolution}).
We would like to prove that
$$
\overline{\mathcal{J}}_\epsilon(w) - \overline{\mathcal{J}}_\epsilon(v_\epsilon) \geq \overline{\mathcal{F}}\, '(v_\epsilon)(w - v_\epsilon), \quad \forall w \in L^{1^*}(\mathbb{R}^N).
$$
To see why, consider $w \in L^{1^*}(\mathbb{R}^N)$ and note that:
\begin{itemize}
\item if $w \in BV(\mathbb{R}^N) \cap L^{1^*}(\mathbb{R}^N)$, then
\begin{eqnarray*}
\overline{\mathcal{J}}_\epsilon(w) - \overline{\mathcal{J}}_\epsilon(v_\epsilon) & = & \mathcal{J}_\epsilon(w) - \mathcal{J}_\epsilon(v_\epsilon)\\
& \geq & \mathcal{F}'(v_\epsilon)(w - v_\epsilon)\\
& = & \int_{\mathbb{R}^N}f(v_\epsilon)(w - v_\epsilon)dx\\
& = & \overline{\mathcal{F}}\, '(v_\epsilon)(w - v_\epsilon);
\end{eqnarray*}

\item if $w \in L^{1^*}(\mathbb{R}^N)\backslash BV(\mathbb{R}^N)$, since $\overline{\mathcal{J}}_\epsilon(w) = +\infty$ and $\overline{\mathcal{J}}_\epsilon(v_\epsilon) < +\infty$, it follows that
\begin{eqnarray*}
\overline{\mathcal{J}}_\epsilon(w) - \overline{\mathcal{J}}_\epsilon(v_\epsilon) & = & +\infty\\
& \geq & \overline{\mathcal{F}}\, '(v_\epsilon)(w - v_\epsilon).
\end{eqnarray*}
\end{itemize}
Therefore the result follows.
\end{proof}

Let us assume that $v_\epsilon \in BV(\mathbb{R}^N)$ is a bounded variation solution of (\ref{Pintrov}), i.e., that $v_\epsilon$ satisfies (\ref{eqsolution}). Since $0 \in \partial \Phi_\epsilon(v_\epsilon)$, by the last result it follows that $0 \in \partial \overline{\Phi}_\epsilon(v_\epsilon)$. Since $\overline{\mathcal{J}}_\epsilon$ is convex and $\overline{\mathcal{F}}$ is smooth, it follows that $\overline{\mathcal{F}}\, '(v_\epsilon) \in \partial \overline{\mathcal{J}}_\epsilon(v_\epsilon)$. Let us define now $\overline{\mathcal{J}_\epsilon^1}(v) := \int_{\mathbb{R}^N} |Dv|$ and $\overline{\mathcal{J}_\epsilon^2}(v) := \int_{\mathbb{R}^N} V(\epsilon x)|v|dx$. Then note that
$$
\overline{\mathcal{F}}\, '(v_\epsilon) \in \partial \overline{\mathcal{J}}_\epsilon(v_\epsilon) \subset \partial \overline{\mathcal{J}_\epsilon^1}(v_\epsilon) + \partial \overline{\mathcal{J}_\epsilon^2}(v_\epsilon).
$$
Then there exist $z_1^*, z_2^* \in L^N(\mathbb{R}^N)$ such that $z_1^* \in \partial \overline{\mathcal{J}_\epsilon^1}(v_\epsilon)$, $z_2^* \in \partial \overline{\mathcal{J}}_\epsilon^2(v_\epsilon)$ and
$$
\overline{\mathcal{F}}\, '(v_\epsilon) = z_1^* +z_2^* \quad \mbox{ in $L^N(\mathbb{R}^N)$.}
$$
Following the same arguments that in \cite[Proposition 4.23, pg. 529]{Kawohl}, we have that there exists $z \in L^\infty(\mathbb{R}^N, \mathbb{R}^N)$ such that $|z|_\infty \leq 1$,
\begin{equation}
-\mbox{div}{z} = z_1^* \quad \mbox{ in $L^N(\mathbb{R}^N)$}
\label{eulerlagrange1}
\end{equation}
and 
\begin{equation}
 -\int_{\mathbb{R}^N}v_\epsilon \mbox{div}z dx = \int_{\mathbb{R}^N}|Dv_\epsilon|,
 \label{eulerlagrange2}
 \end{equation}
where the divergence in (\ref{eulerlagrange1}) has to be understood in the distributional sense. Moreover, the same result implies that $z_2^*$ is such that 
\begin{equation}
z_2^* V(\epsilon x)|v_\epsilon| = v_\epsilon, \quad \mbox{a.e. in $\mathbb{R}^N$.}
\label{eulerlagrange3}
\end{equation}
Therefore, it follows from (\ref{eulerlagrange1}), (\ref{eulerlagrange2}) and (\ref{eulerlagrange3}) that $v_\epsilon$ satisfies
\begin{equation}
\left\{
\begin{array}{l}
\exists z \in L^\infty(\mathbb{R}^N,\mathbb{R}^N), \, \, \|z\|_\infty \leq 1,\, \,  \mbox{div}z \in L^N(\mathbb{R}^N), \, \, -\int_{\mathbb{R}^N}v_\epsilon \mbox{div}z dx = \int_{\mathbb{R}^N}|Dv_\epsilon|,\\
\exists z_2^* \in L^N(\mathbb{R}^N),\, \, z_2^*V(\epsilon x)|v_\epsilon| = v_\epsilon \quad \mbox{a.e. in $\mathbb{R}^N$},\\
-\mbox{div} z + z_2^* = f(v_\epsilon), \quad \mbox{a.e. in $\mathbb{R}^N$}.
\end{array}
\right.
\label{eulerlagrangeequation}
\end{equation}

Hence, (\ref{eulerlagrangeequation}) is the precise version of (\ref{Pintro}).

\section{Existence and concentration of solution with the Rabinowitz's condition}
\label{sectionexistence}

Let us first observe that by standard calculations, it is possible to prove that $\Phi_\epsilon$, $\Phi_\infty$ and $\Phi_0$ satisfy the geometrical conditions of the Mountain Pass Theorem. Then the following minimax levels are well defined
$$
c_\epsilon = \inf_{\gamma \in \Gamma_\epsilon}\sup_{t \in [0,1]}\Phi_\epsilon(\gamma(t)),
$$
$$
c_\infty = \inf_{\gamma \in \Gamma_\infty}\sup_{t \in [0,1]}\Phi_\infty(\gamma(t))
$$
and
$$
c_0 = \inf_{\gamma \in \Gamma_0}\sup_{t \in [0,1]}\Phi_0(\gamma(t)),
$$
where $\Gamma_\epsilon = \{\gamma \in C([0,1],BV(\mathbb{R}^N); \, \gamma(0) = 0 \, \mbox{and} \, \Phi_\epsilon(\gamma(1)) < 0\}$  and  $\Gamma_\infty, \Gamma_0$ are defined in an analogous way. Moreover, by study made in Section 5, it follows that there exists a critical point of $\Phi_\infty$, $w_\infty \in BV(\mathbb{R}^N)$, such that $\Phi_\infty(w_\infty) = c_\infty$. By the same reason, there exists a critical point of $\Phi_0$, $w_0 \in BV(\mathbb{R}^N)$, such that $\Phi_0(w_0) = c_0$.

Let us define the Nehari manifolds associated to $\Phi_\epsilon$, $\Phi_\infty$ and $\Phi_0$, which are well defined by (\ref{derivadaJ}), respectively by
$$
\mathcal{N}_\epsilon = \{v \in BV(\mathbb{R}^N)\backslash\{0\}; \, \Phi_\epsilon'(v)v = 0\},
$$
$$
\mathcal{N}_\infty = \{v \in BV(\mathbb{R}^N)\backslash\{0\}; \, \Phi_\infty'(v)v = 0\}
$$
and
$$
\mathcal{N}_0 = \{v \in BV(\mathbb{R}^N)\backslash\{0\}; \, \Phi_0'(v)v = 0\}.
$$

By the discussion in \cite{FigueiredoPimenta1}, it follows that $c_\epsilon = \inf_{\mathcal{N}_\epsilon}\Phi_\epsilon$, $c_\infty = \inf_{\mathcal{N}_\infty}\Phi_\infty$ and $c_0 = \inf_{\mathcal{N}_0}\Phi_0$.

\subsection{Existence results}

First of all we study the behavior of the minimax levels $c_\epsilon$, when $\epsilon \to 0^+$. For the sake of simplicity, let us suppose without lack of generality that $V(0) = V_0$.

\begin{lemma}
$\displaystyle \lim_{\epsilon \to 0^+} c_\epsilon = c_0.$
\label{lemmac_0}
\end{lemma}
\begin{proof}
Let $\epsilon \to 0$ as $n \to +\infty$. Let $\psi\in C^\infty_0(\mathbb{R}^N)$ be such that $0\leq
\psi \leq 1$, $\psi \equiv 0$ in $\mathbb{R}^N\slash B_2(0)$, $\psi \equiv 1$ in
$B_1(0)$ and $|\nabla \psi|\leq C$ in $\mathbb{R}^N$. Let us define 
$$w_{\epsilon_n}(x) = \psi(\epsilon_n x)w_0(x),$$
where $w_0$ is the ground state critical point of $\Phi_0$.
Note that $w_{\epsilon_n} \rightarrow w_0$ in $BV(\mathbb{R}^N)$ and $\Phi_0(w_{\epsilon_n}) \rightarrow \Phi_0(w)$ as $n \to +\infty$.
Let $t_{\epsilon_n}$ be such that $t_{\epsilon_n} w_{\epsilon_n} \in \mathcal{N}_{\epsilon_n}$ and let us suppose just for a while that $t_{\epsilon_n} \rightarrow 1$ as $n \to +\infty$. Then
\begin{eqnarray*}
c_{\epsilon_n} & \leq & \Phi_{\epsilon_n}(t_{\epsilon_n} w_{\epsilon_n})\\
& = & \Phi_0(t_{\epsilon_n} w_{\epsilon_n}) +
\int_{\mathbb{R}^N}\left(V({\epsilon_n}
x) - V_0\right)t_{\epsilon_n} |w_{\epsilon_n}| dx.
\end{eqnarray*}

Using the Lebesgue Dominated Theorem, it follows that
$$
\limsup_{n \to +\infty} c_{\epsilon_n} \leq \Phi_0(w_0) = c_0.
$$

On the other hand, since $\Phi_0(v) \leq \Phi_{\epsilon_n}(v)$ for all $v \in BV(\mathbb{R}^N)$, it follows that $c_0 \leq c_{\epsilon_n}$. Then
$$
\lim_{n \to +\infty} c_{\epsilon_n} = c_0.
$$

What is left to do is to prove that in fact $t_{\epsilon_n} \to 1$, as $n \to +\infty$. 
Since $\Phi_{\epsilon_n}'(t_{\epsilon_n} w_{\epsilon_n})w_{\epsilon_n} = 0$, it follows that
$$
t_{\epsilon_n} \left( \int_{\mathbb{R}^N}|Dw_{\epsilon_n}| + \int_{\mathbb{R}^N}V({\epsilon_n} x)|w_{\epsilon_n}| dx \right) = \int_{\mathbb{R}^N}f(t_{\epsilon_n} w_{\epsilon_n})w_{\epsilon_n} dx.
$$
We claim that $(t_{\epsilon_n})_{{\epsilon_n} > 0}$ is bounded. In fact, on the contary, up to a subsequence, $t_{\epsilon_n} \rightarrow +\infty$. Let $\Sigma \subset \mathbb{R}^N$ be such that $|\Sigma| > 0$ and $w_0(x) \neq 0$ for all $x\in \Sigma$. Hence it holds for all $n \in \mathbb{N}$ that
\begin{eqnarray*}
\|w_{\epsilon_n}\|_{\epsilon_n} & = & \int_{\mathbb{R}^N}\frac{f(t_{\epsilon_n}w_{\epsilon_n})t_{\epsilon_n}w_{\epsilon_n}}{t_{\epsilon_n}}dx\\
& \geq &  \int_{\Sigma}\frac{\theta F(t_{\epsilon_n}w_{\epsilon_n})}{t_{\epsilon_n}}dx.
\end{eqnarray*}
Then by $(f_4)$ and Fatou's Lemma it follows that 
$$
\|w_{\epsilon_n}\|_{\epsilon_n} \rightarrow +\infty, \quad \mbox{as $n\to\infty$},
$$ which contradicts the fact that $w_{\epsilon_n} \to w_0$ in $BV(\mathbb{R}^N)$ as $n \to \infty$.

Now we have to verify that $t_{\epsilon_n} \not \to 0$ as $n \to +\infty$. In fact, on the contrary, from $(f_2)$ and the fact that $t_{\epsilon_n}w_{\epsilon_n} \in \mathcal{N}_{\epsilon_n}$, we would have that
$$
\|w_{\epsilon_n}\|_{\epsilon_n} = \int_{\mathbb{R}^N} f(t_{\epsilon_n}w_{\epsilon_n})w_{\epsilon_n} dx =  o_n(1),
$$
a clear contradiction.
Then there exist $\alpha,\beta >0$ such that 
$$
\alpha \leq t_{\epsilon_n} \leq \beta \quad \mbox{for all $n \in \mathbb{N}$}
$$
and then, up to a subsequence, $t_n \to \overline{t} > 0$, as $n \to +\infty$. Since $w_{\epsilon_n} \to w_0$ in $BV(\mathbb{R}^N)$, from the definition of $w_0$, it follows by $(f_5)$ that $\overline{t} = 1$.
\end{proof}

Since by $(V_2)$, $V_0 < V_\infty$, it follows from the monotonicity of the energy functional w.r.t. the potentials that
\begin{equation}
c_0 < c_\infty.
\label{c_0c_infty}
\end{equation}

As a consequence of Lemma \ref{lemmac_0} and (\ref{c_0c_infty}), it holds the following result.

\begin{corollary}
There exists $\epsilon_0 > 0$ such that $c_\epsilon < c_\infty$ for all $\epsilon \in (0,\epsilon_0)$.
\label{corolariocepsiloncinfty}
\end{corollary}

By \cite[Theorem 4]{FigueiredoPimenta2}, for each $\epsilon > 0$, there exists a Palais-Smale sequence $(v_n) \subset BV(\mathbb{R}^N)$ to $\Phi_\epsilon$ in the level $c_\epsilon$, i.e.
\begin{equation}
\lim_{n \to \infty}\Phi_\epsilon(v_n) = c_\epsilon
\label{Phivn}
\end{equation}
and
\begin{equation}
\|w\|_\epsilon - \|v_n\|_\epsilon \geq \int_{\mathbb{R}^N}f(v_n)(w - v_n)dx - \tau_n \|w - v_n\|_\epsilon, \quad \forall w \in BV(\mathbb{R}^N),
\label{PSepsilon}
\end{equation}
where $\tau_n \to 0$, as $n \to \infty$.

\begin{lemma}
The sequence $(v_n)$ is bounded in $BV(\mathbb{R}^N)$.
\label{lemmabounded}
\end{lemma}
\begin{proof}
Let us consider $w = 2v_n$ in (\ref{PSepsilon}) and note that
$$
\|v_n\|_\epsilon \geq \int_{\mathbb{R}^N}f(v_n)v_ndx - \tau_n\|v_n\|_\epsilon,
$$
which implies that
\begin{equation}
(1 + \tau_n)\|v_n\|_\epsilon \geq \int_{\mathbb{R}^N}f(v_n)v_ndx.
\label{lemmaboundedeq1}
\end{equation}
Then, by $(f_4)$ and (\ref{lemmaboundedeq1}), 
\begin{eqnarray*}
c_\epsilon + o_n(1) & \geq & \Phi_\epsilon(v_n)\\
& = & \|v_n\|_\epsilon + \int_{\mathbb{R}^N}\left(\frac{1}{\theta}f(v_n)v_n - F(v_n)\right)dx - \int_{\mathbb{R}^N}\frac{1}{\theta}f(v_n)v_ndx\\
& \geq & \|v_n\|_\epsilon \left(1- \frac{1}{\theta} - \frac{\tau_n}{\theta}\right)\\
& \geq & C\|v_n\|_\epsilon,
\end{eqnarray*}
for some $C > 0$ which does not depend on $n \in \mathbb{N}$. Then the result follows.
\end{proof}

By the last result and the compactness of the embeddings of $BV(\mathbb{R}^N)$ in $L^q_{loc}(\mathbb{R}^N)$ for $1 \leq q < 1^*$, it follows that there exists $v_\epsilon \in BV_{loc}(\mathbb{R}^N)$ such that 
\begin{equation}
v_n \to v_\epsilon \quad \mbox{in $L^q_{loc}(\mathbb{R}^N)$ for $1 \leq q < 1^*$}
\label{convergencevn}
\end{equation}
and
$$
v_n \to v_\epsilon \quad \mbox{a.e. in $\mathbb{R}^N$,}
$$
as $n \to +\infty$. Note that $v_\epsilon \in BV(\mathbb{R}^N)$. In fact, if $R > 0$, by the semicontinuity of the norm in $BV(B_R(0))$ w.r.t. the $L^1(B_R(0))$ topology it follows that
\begin{equation}
\|v_\epsilon\|_{BV(B_R(0))} \leq \liminf_{n \to +\infty}\|v_n\|_{BV(B_R(0))} \leq \liminf_{n \to +\infty}\|v_n\|_{BV(\mathbb{R}^N)} \leq C,
\label{vBV}
\end{equation}
where $C$ does not depend on $n$ or $R$.
Since the last inequality holds for every $R > 0$, then $v_\epsilon \in BV(\mathbb{R}^N)$.

The following is a crucial result in our argument. In its proof we use the well known {\it Concentration of Compactness Principle} of Lions \cite{Lions}.

\begin{proposition}
If $\epsilon < \epsilon_0$, $\epsilon_0$ like in Corollary \ref{corolariocepsiloncinfty}, then
\begin{equation}
v_n \to v_\epsilon \quad \mbox{in $L^q(\mathbb{R}^N)$ for all $1 \leq q < 1^*$.}
\label{convergence}
\end{equation}
\label{propositionconvergence}
\end{proposition}
\begin{proof}
Let us apply the Concentration of Compactness Principle of Lions to the following bounded sequence in $L^1(\mathbb{R}^N)$,
$$
\rho_n(x):= \frac{|v_n(x)|}{|v_n|_1}.
$$

For future reference, note that
\begin{equation}
|v_n|_1 \not \to 0, \quad \mbox{as $n \to +\infty$.}
\label{equ_n0}
\end{equation}
In fact, otherwise, by the boundedness of $(v_n)$ in $L^{1^*}(\mathbb{R}^N)$, by interpolation inequality $(v_n)$ would converge to $0$ in $L^q(\mathbb{R}^N)$ for all $1 \leq q < 1^*$. By taking $w = v_n +tv_n$ in (\ref{PSepsilon}) and doing $t \to 0$, it is easy to see that
$$
\|v_n\|_\epsilon = \int_{\mathbb{R}^N}f(v_n)v_ndx + o_n(1).
$$
Then, by the last equality, $(f_2)$, $(f_3)$ and the fact that $v_n \to 0$ in $L^q(\mathbb{R}^N)$ for all $1 \leq q < 1^*$, the Lebesgue Convergence Theorem imply that $v_n \to 0$ in $BV(\mathbb{R}^N)$, implying that $c_\epsilon = 0$, which contradicts the fact that $c_\epsilon > 0$.

Since  $(\rho_n)$ is a bounded sequence in $L^1(\mathbb{R}^N)$, the {\it Concentration of Compactness Principle} implies that one and only one of the following statements holds:
\begin{description}
\item [{\it (Vanishing)}] $\displaystyle \lim_{n \to +\infty}\sup_{y \in \mathbb{R}^N} \int_{B_R(y)}\rho_n dx = 0$, $\forall R >0$;
\item [{\it (Compactness)}] There exist $(y_n) \subset \mathbb{R}^N$ such that for all $\eta > 0$, there exists $R > 0$ such that
\begin{equation}
\int_{B_R(y_n)}\rho_n dx \geq 1 - \eta, \quad \forall n \in \mathbb{N};
\label{compactness}
\end{equation}
\item [{\it (Dichotomy)}] There exist $(y_n) \subset \mathbb{R}^N$, $\alpha \in (0,1)$, $R_1 > 0$, $R_n \to +\infty$ such that the functions $\displaystyle \rho_{n,1}(x) := \chi_{B_{R_1}(y_n)}(x)\rho_n(x)$ and $\displaystyle \rho_{n,2}(x) := \chi_{B_{R_n}^c(y_n)}(x)\rho_n(x)$ satisfy
\begin{equation}
\int_{\mathbb{R}^N}\rho_{1,n} dx \to \alpha \quad \mbox{and} \quad \int_{\mathbb{R}^N}\rho_{2,n} dx \to 1 - \alpha.
\label{dichotomy}
\end{equation}
\end{description}

Our objective is to show that $(\rho_n)$ verifies the {\it Compactness} condition and in order to do so we act by excluding all the others possibilities. 

Note that {\it Vanishing} does not occur. In fact, otherwise, by \cite[Theorem 1.1]{FigueiredoPimenta2}, it would hold that $\rho_n \to 0$ in $L^q(\mathbb{R}^N)$, for all $1 \leq q < 1^*$. Taking (\ref{equ_n0}) into account, this would imply that $v_n \to 0$ in $L^q(\mathbb{R}^N)$, for all $1 \leq q < 1^*$ and then, this would led us to $c_\epsilon = 0$, a clear contradiction.

Let us show now that {\it Dichotomy} also does not hold. Firstly note that (\ref{PSepsilon}) implies that
\begin{equation}
\Phi_\epsilon'(v_n)v_n = o_n(1), \quad \mbox{as $n \to +\infty$.}
\end{equation}

As far as the sequence $(y_n)$ is concerned, let us consider the two possible situations.
\begin{itemize}
\item $(y_n)$ is bounded:

In this case the function $v_\epsilon$ is nontrivial, since 
$$
\int_{B_R(y_n)}\frac{|v_n|}{|v_n|_1}dx \to \alpha,
$$
implies that
$$
\int_{B_R(y_n)}|v_n|dx \geq \delta, \quad \mbox{for all $n$ sufficiently large.}
$$
Then, by taking $R_0 > 0$ such that $B_R(y_n) \subset B_{R_0}(0)$ for all $n \in \mathbb{N}$, it follows that
$$
\int_{B_{R_0}(0)}|v_n|dx \geq \delta, \quad \mbox{for all $n$ sufficiently large,}
$$
implying by (\ref{convergencevn}) that
\begin{equation}
\int_{B_{R_0}(0)}|v_\epsilon|dx \geq \delta.
\label{vnontrivial}
\end{equation}

Now let us show the following claim.
\begin{claim}
$\displaystyle \Phi_\epsilon'(v_\epsilon)v_\epsilon \leq 0.$
\label{derivadanegativa}
\end{claim}

Note that, if $\varphi \in C^\infty_0(\mathbb{R}^N)$, $0 \leq \varphi \leq 1$, $\varphi \equiv 1$ in $B_R(0)$ and $\varphi \equiv 0$ in $B_{2R}(0)^c$, for $\varphi_R := \varphi (\cdot/R)$, it follows that for all $v \in BV(\mathbb{R}^N)$, 
\begin{equation}
(D(\varphi_R v))^s \quad \mbox{is absolutely continuous w.r.t.} \quad (Dv)^s.
\label{derivadafuncaoteste}
\end{equation} 
In fact, note that
$$
D(\varphi_R v) = \nabla \varphi_R v + \varphi_R Dv = \nabla \varphi_R v + \varphi_R Dv^a + \varphi_R Dv^s , \quad \mbox{in $\mathcal{D}'(\mathbb{R}^N)$}.
$$
Then it follows that 
$$
(D(\varphi_R v))^s = (\varphi_R (Dv)^s)^s = \varphi_R (Dv)^s.
$$

Taking (\ref{derivadafuncaoteste}) into account, the fact that $\varphi_R v_n$ is equal to $0$ a.e. in the set where $v_n$ vanishes and also the fact that $\frac{\varphi_R \mu}{|\varphi_R \mu|} = \frac{\mu}{|\mu|}$ a.e. in $B_R(0)$, it is well defined $\Phi_\epsilon'(v_n)(\varphi_R v_n)$ and, by (\ref{Jlinha}), it follows that
\begin{eqnarray*}
\Phi_\epsilon'(v_n)(\varphi_R v_n) & = & \int_{\mathbb{R}^N}\frac{((Dv_n)^a)^2 \varphi_R + v_n(Dv_n)^a \cdot \nabla \varphi_R}{|(Dv_n)^a|}dx \\
& & + \int_{\mathbb{R}^N}\frac{Dv_n}{|Dv_n|}\frac{\varphi_R(Dv_n)^s}{|\varphi_R (Dv_n)^s|}|\varphi_R(D v_n)^s| + \\
& & + \int_{\mathbb{R}^N}V(\epsilon x) \mbox{sgn}(v_n)(\varphi_R v_n)dx - \int_{\mathbb{R}^N}f(v_n)\varphi_R v_ndx\\
& = & \int_{\mathbb{R}^N}\varphi_R |(Dv_n)^a|dx + \int_{\mathbb{R}^N}\frac{v_n(Dv_n)^a \cdot \nabla\varphi_R}{|(Dv_n)^a|}dx +\\
& & + \int_{\mathbb{R}^N}\frac{(Dv_n)^s}{|(Dv_n)^s|}\frac{\varphi_R(Dv_n)^s}{|\varphi_R (Dv_n)^s|}|\varphi_R(D v_n)^s| + \int_{\mathbb{R}^N}V(\epsilon x) |v_n|\varphi_R dx -\\
& & - \int_{\mathbb{R}^N}f(v_n)\varphi_R v_ndx.
\end{eqnarray*}
The last inequality together with the lower semicontinuity of the norm in $BV(B_R(0))$ w.r.t. the $L^1(B_R(0))$ convergence and the fact that $\Phi_\epsilon'(v_n)(\varphi_R v_n) = o_n(1)$ (since $(\varphi_R v_n)$ is bounded in $BV(\mathbb{R}^N))$, imply that
\begin{equation}
\int_{B_R(0)}|Dv_\epsilon| + \liminf_{n \to \infty}\int_{\mathbb{R}^N}\frac{v_n(Dv_n)^a \cdot \nabla\varphi_R}{|(Dv_n)^a|}dx + \int_{\mathbb{R}^N}V(\epsilon x)\varphi_R|v_\epsilon|dx \leq \int_{\mathbb{R}^N}f(v_\epsilon)v_\epsilon \varphi_Rdx.
\label{eqv}
\end{equation}
By doing $R \to +\infty$ in both sides of (\ref{eqv}) we get that
\begin{equation}
\int_{\mathbb{R}^N}|Dv_\epsilon| + \int_{\mathbb{R}^N}V(\epsilon x)|v_\epsilon|dx \leq \int_{\mathbb{R}^N}f(v_\epsilon)v_\epsilon dx
\label{inequalityv}
\end{equation}
what proves the claim.

By the Claim and (\ref{vnontrivial}), it follows that there exists $t_\epsilon \in (0,1]$ such that $t_\epsilon v_\epsilon \in \mathcal{N}_\epsilon$. 

Note also that
\begin{equation}
c_\epsilon + o_n(1) = \Phi_\epsilon(v_n) +o_n(1) = \Phi_\epsilon(v_n) - \Phi'_\epsilon(v_n)v_n = \int_{\mathbb{R}^N}\left(f(v_n)v_n - F(v_n)\right) dx.
\label{convergenciaf}
\end{equation}

Then applying Fatou Lemma in the last inequality together with $(f_4)$, it follows that
\begin{eqnarray*}
c_\epsilon & \geq &  \int_{\mathbb{R}^N}\left(f(v_\epsilon)v_\epsilon - F(v_\epsilon)\right) dx\\
& \geq & \int_{\mathbb{R}^N}\left(f(t_\epsilon v_\epsilon)t_\epsilon v_\epsilon - F(t_\epsilon v_\epsilon)\right) dx\\
& = & \Phi_\epsilon(t_\epsilon v_\epsilon) - \Phi'_\epsilon(t_\epsilon v_\epsilon)t_\epsilon v_\epsilon\\
& = & \Phi_\epsilon(t_\epsilon v_\epsilon)\\
& \geq & c_\epsilon.
\end{eqnarray*}

Hence, $t_\epsilon=1$ and $\Phi_{\epsilon}(v_\epsilon)=c_\epsilon$. This together with (\ref{convergenciaf}) and $(f_4)$ yield
$$
f(v_n)v_n \to f(v_\epsilon)v_\epsilon \quad \mbox{in $L^1(\mathbb{R}^N)$}
$$
$$
F(v_n) \to F(v_\epsilon) \quad \mbox{in $L^1(\mathbb{R}^N)$}
$$
and
$$
\|v_n\|_\epsilon \to \|v_\epsilon\|_\epsilon, 
$$
from where it follows that  
$$
v_n \to v_\epsilon \quad \mbox{in $L^1(\mathbb{R}^N)$.}
$$

Here, we have used the fact that $(f_4)$ implies that 
$$
(1-\theta)f(\tilde{v}_n)\tilde{v}_n \leq f(\tilde{v}_n)\tilde{v}_n - F(\tilde{v}_n),
$$
then by applying the Lebesgue Dominated Convergence Theorem, it follows that 
$$
f(v_n)v_n \to f(\tilde{v})\tilde{v} \quad \mbox{in $L^1(\mathbb{R})$.}
$$

As a consequence, since $(y_n)$ is a bounded sequence and $R_n \to +\infty$,  the $L^1(\mathbb{R}^N)$ convergence of $(v_n)$ leads to
\begin{equation}
\int_{B_{R_n}^c(y_n)}|v_n| dx \to 0 \quad \mbox{as $n \to +\infty$.}
\label{convergenciavn}
\end{equation}
On the other hand, since $v_n \to v_\epsilon \neq 0$ in $L^1(\mathbb{R}^N)$ and by (\ref{dichotomy}), it follows that
$$
\int_{B_{R_n}^c(y_n)}|v_n| dx \to (1-\alpha) |v_\epsilon|_{L^1(\mathbb{R}^N)} > 0, \quad \mbox{as $n \to +\infty$,}
$$
a clear contradiction with (\ref{convergenciavn}).
\item $(y_n)$ is unbounded:

In this case we should proceed as in the case where $(y_n)$ were bounded, but now dealing with the sequence $(\tilde{v}_n)$ where $\tilde{v}_n = v_n(\cdot - y_n)$. In fact, since $\|v_n\|_{BV(\mathbb{R}^N)} = \|\tilde{v}_n\|_{BV(\mathbb{R}^N)}$, it follows that $(\tilde{v}_n)$ is bounded and then converges, up to a subsequence, to some function $\tilde{v} \in BV(\mathbb{R}^N)$ in $L^1_{loc}(\mathbb{R}^N)$, where $\tilde{v} \neq \equiv 0$ by (\ref{dichotomy}).

\begin{claim}
\begin{equation}
\displaystyle \Phi_\infty'(\tilde{v})\tilde{v} \leq 0.
\label{derivadanegativaIinfty}
\end{equation}
\label{Claimunbounded}
\end{claim}

In order to prove this claim, let us denote,  for $v \in BV(\mathbb{R}^N)$,
$$
\|v\|_{\epsilon,y_n} := \int_{\mathbb{R}^N}|Dv| + \int_{\mathbb{R}^N}V(\epsilon x + \epsilon y_n)|v|dx
$$
and
$$
\Phi_{\epsilon,y_n}(v) := \|v\|_{\epsilon,y_n} - \int_{\mathbb{R}^N}F(v)dx.
$$
Note that, as before, $\Phi_{\epsilon,y_n}'(v)w$ is well defined for all $v,w \in BV(\mathbb{R}^N)$ such that $(Dw)^s$ is absolutely continuous w.r.t. $(Dv)^s$ and $w$ is equal to $0$ a.e. in the set where $v$ vanishes. Moreover, 
\begin{equation}
\begin{array}{c}
\displaystyle \Phi_{\epsilon,y_n}'(v)w = \int_{\mathbb{R}^N} \frac{(Dv)^a(Dw)^a}{|(Dv)^a|}dx + \int_{\mathbb{R}^N} \frac{Dv}{|Dv|}(x)\frac{Dw}{|Dw|}(x)|(Dw)|^s \\
\displaystyle + \int_{\mathbb{R}^N}V(\epsilon x + \epsilon y_n)\mbox{sgn}(v) w dx - \int_{\mathbb{R}^N}f(v)wdx.
\end{array}
\label{Phiynderivada}
\end{equation}

Since $\displaystyle \int_{\mathbb{R}^N}|Dv_n| = \int_{\mathbb{R}^N}|D\tilde{v}_n|$, from (\ref{PSepsilon}), by a change of variable, for all $w \in BV(\mathbb{R}^N)$ we have that
\begin{equation}
\begin{array}{ll}
\|w(.+y_n)\|_{\epsilon,y_n} - \|\tilde{v}_n\|_{\epsilon,y_n} & \geq  \displaystyle \int_{\mathbb{R}^N}f(\tilde{v}_n)(w(.+y_n) - \tilde{v}_n)dx\\
&  \displaystyle - \tau_n \left(\int_{\mathbb{R}^N}|D(w(\cdot + y_n) - v_n)| + \int_{\mathbb{R}^N}V(\epsilon x + \epsilon y_n)|w(.+y_n) - \tilde{v}_n|dx \right),
\end{array}
\label{eqsolutiontranslation}
\end{equation}
For $\varphi_R := \varphi(\cdot/R)$, where $\varphi \in C^\infty_0(\mathbb{R}^N)$, $0 \leq \varphi \leq 1$, $\varphi \equiv 1$ in $B_R(0)$ and $\varphi \equiv 0$ in $B_{2R}(0)^c$, by taking in (\ref{eqsolutiontranslation}) $w(x) = v_n(x) + t\varphi_R(x-y_n)v_n(x)$ and making $t \to 0$, it follows that

\begin{equation}
\Phi_{\epsilon,y_n}'(\tilde{v}_n)(\varphi_R \tilde{v}_n) = o_n(1).
\label{Phiynon1}
\end{equation}
From (\ref{Phiynon1}), proceeding as in (\ref{eqv}) and taking into account that $|y_n| \to +\infty$, we get that
\begin{equation}
\int_{\mathbb{R}^N}|D\tilde{v}| + \int_{\mathbb{R}^N}V_\infty |\tilde{v}|dx \leq \int_{\mathbb{R}^N}f(\tilde{v})\tilde{v} dx,
\label{inequalityvtilde}
\end{equation}
which proves the claim.

By the Claim \ref{Claimunbounded} and since $\tilde{v} \neq 0$, it follows that there exists $\tilde{t} \in (0,1]$ such that $\tilde{t} v \in \mathcal{N}_\infty$. 
On the other hand, recalling that 
Note that
\begin{eqnarray*}
c_\epsilon + o_n(1) & = & \Phi_\epsilon(v_n) + o_n(1) \\
& = & \Phi_\epsilon(v_n) - \Phi'_\epsilon(v_n)v_n\\
& = & \int_{\mathbb{R}^N}\left(f(v_n)v_n - F(v_n)\right) dx\\
& = & \int_{\mathbb{R}^N}\left(f(\tilde{v}_n)\tilde{v}_n - F(\tilde{v}_n)\right) dx.
\end{eqnarray*}
the Fatou's Lemma gives 
\begin{eqnarray*}
c_\epsilon & \geq & \int_{\mathbb{R}^N}\left(f(\tilde{v})\tilde{v} - F(\tilde{v})\right) dx\\
& \geq & \int_{\mathbb{R}^N}\left(f(\tilde{t}\tilde{v})\tilde{t}\tilde{v} - F(\tilde{t}\tilde{v})\right) dx\\
& = & \Phi_\infty(\tilde{t}\tilde{v})\\
& \geq & c_\infty,
\end{eqnarray*}
which contradicts Corollary \ref{corolariocepsiloncinfty} when $\epsilon$ is small enough.

Then we can conclude that {\it Dichotomy} in fact does not happen and then, it follows that {\it Compactness} must hold.

\begin{claim}
$(y_n)$ in (\ref{compactness}) is a bounded sequence in $\mathbb{R}^N$.
\label{claimconcentration}
\end{claim}

Assuming this claim, for $\eta > 0$, there exists $R > 0$ such that, by (\ref{compactness}),
$$
\int_{B_R^c(0)}\rho_n dx < \eta, \quad \forall n \in \mathbb{N},
$$
which is equivalent to
\begin{equation}
\int_{B_R^c(0)}|v_n| dx \leq \eta|v_n|_1 \leq C\eta, \quad \forall n \in \mathbb{N}.
\label{decay1}
\end{equation}
Since $v_\epsilon \in L^1(\mathbb{R}^N)$, there exists $R_0 > 0$ such that
\begin{equation}
\int_{B_{R_0}^c(0)}|v_\epsilon| dx \leq \eta.
\label{decay2}
\end{equation}
Then, for $R_1 \geq \max\{R,R_0\}$, since $v_n \to v_\epsilon$ in $L^1(B_{R_1}(0))$, there exists $n_0 \in \mathbb{N}$ such that
\begin{equation}
\int_{B_{R_1}(0)}|v_n - v_\epsilon| dx \leq \eta \quad \forall n \geq n_0.
\label{decay3}
\end{equation}
Then, by (\ref{decay1}), (\ref{decay2}) and (\ref{decay3}), it follows that if $n \geq n_0$,
$$
\int_{\mathbb{R}^N}|v_n - v_\epsilon| dx \leq \eta + \int_{B_{R_1}^c(0)}|v_n - v_\epsilon| dx \leq \eta + \int_{B_{R_1}^c(0)}|v_n| dx + \int_{B_{R_1}^c(0)}|v_\epsilon| dx \leq C_1\eta.
$$
Hence $v_n \to v_\epsilon$ in $L^1(\mathbb{R}^N)$ and since $(v_n)$ is bounded in $L^{1^*}(\mathbb{R}^N)$, by interpolation inequality it follows that
$$
v_n \to v_\epsilon \quad \mbox{in $L^q(\mathbb{R}^N)$, for all $1 \leq q < 1^*$.}
$$

Now, what is left is proving Claim \ref{claimconcentration}. However, the proof of it consist in suppose by contradiction that, up to a subsequence, $|y_n| \to +\infty$ and then proceed as in the case of {\it Dichotomy}, where $(y_n)$ were unbounded, reaching that $c_\epsilon \geq c_\infty$. But the latter is a clear contradiction when $\epsilon < \epsilon_0$, in the light of Corollary \ref{corolariocepsiloncinfty}. 

\end{itemize}
\end{proof}

Now let us just remark that if $\epsilon < \epsilon_0$, then $v_\epsilon$ is in fact a nontrivial solution of (\ref{Pintrov}). First of all note that (\ref{convergence}), $(f_2)$ and $(f_3)$ implies that
\begin{equation}
\int_{\mathbb{R}^N}f(v_n)v_ndx \to \int_{\mathbb{R}^N}f(v_\epsilon)v_\epsilon dx, \quad \mbox{as $n \to +\infty$.}
\label{convergencef1}
\end{equation}

Then from (\ref{PSepsilon}), (\ref{convergencef1}) and the lower semicontinuity of $\|\cdot \|_\epsilon$ w.r.t. the $L^1(\mathbb{R}^N)$ convergence imply that
\begin{equation}
\|w\|_\epsilon - \|v_\epsilon\|_\epsilon \geq \int_{\mathbb{R}^N}f(v_\epsilon)(w - v_\epsilon)dx, \quad \forall w \in BV(\mathbb{R}^N),
\label{solutionepsilon}
\end{equation}
and then $v_\epsilon$ is in fact a nontrivial solution of (\ref{Pintrov}). Moreover, note that from (\ref{Phivn})
\begin{eqnarray*}
c_\epsilon & \leq & \Phi_\epsilon (v_\epsilon)\\
& = & \Phi_\epsilon (v_\epsilon) - \Phi_\epsilon'(v_\epsilon)v_\epsilon\\
& = & \int_{\mathbb{R}^N}\left(f(v_\epsilon)v_\epsilon - F(v_\epsilon)\right)dx\\
& \leq & \liminf_{n \to \infty}\int_{\mathbb{R}^N}\left(f(v_n)v_n - F(v_n)\right)dx\\
& = & \Phi_\epsilon(v_n) + o_n(1)\\
& = & c_\epsilon.
\end{eqnarray*}
Then $v_\epsilon$ is a ground-state solution of (\ref{Pintrov}) and consequently $u_\epsilon = v_\epsilon(\cdot/\epsilon)$ is a ground-state bounded variation solution of (\ref{Pintro}).

\subsection{Concentration behavior}

In the last section we have proved that for each $\epsilon \in (0, \epsilon_0)$, there exists $v_\epsilon \in BV(\mathbb{R}^N)$ solution of (\ref{Pintrov}) such that $\Phi_\epsilon(v_\epsilon) = c_\epsilon$. Now let us show that this sequence of solutions concentrate around a global minimum of $V$. Before it, let us state and prove some preliminaries lemmas.

\begin{lemma}
There exist $\{y_\epsilon\}_{\epsilon > 0} \subset \mathbb{R}^N$ and $R, \delta > 0$ such that
$$
\liminf_{\epsilon \to 0}\int_{B_R(y_\epsilon)}|v_\epsilon| dx \geq \delta > 0.
$$
\label{lemalions}
\end{lemma}
\begin{proof}
In fact, on the contrary, thanks to \cite[Theorem 1]{FigueiredoPimenta2}, it follows that $v_\epsilon \to 0$ in $L^q(\mathbb{R}^N)$ for all $1 \leq q < 1^*$, as $\epsilon \to 0$. Then, by $(f_2)$, $(f_3)$ and the Lebesgue Convergence Theorem, it follows that 
$$
\displaystyle \int_{\mathbb{R}^N}f(v_\epsilon)v_\epsilon dx = o_\epsilon(1).
$$ 

Taking $w = v_\epsilon \pm t v_\epsilon$ in (\ref{eqsolution}) and passing to the limit as $t \to 0^+$, it follows that 
$$
\|v_\epsilon\|_\epsilon = \int_{\mathbb{R}^N}f(v_\epsilon)v_\epsilon dx = o_\epsilon(1),
$$
which implies that $c_\epsilon = \Phi_\epsilon(v_\epsilon) = o_\epsilon(1)$, leading to a contradiction with Lemma \ref{lemmac_0}.
\end{proof}

\begin{lemma}
The set $\{\epsilon \, y_\epsilon\}_{\epsilon > 0}$ is bounded in $\mathbb{R}^N$.
\label{yepsilonlimitada}
\end{lemma}
\begin{proof}
Suppose by contradiction that there exist $\epsilon_n \to 0$, such that $|\epsilon_n y_n| \to \infty$, as $n \to \infty$, where $y_n:= y_{\epsilon_n}$. In the following we proceed as in the proof of Claim \ref{Claimunbounded} of Proposition \ref{propositionconvergence}. Let $v_n :=v_{\epsilon_n}$, and note that, if $\varphi_R$ is like in the proof of such claim, it follows that
$$
\Phi'_{\epsilon_n,y_n}(\tilde{v}_n)(\varphi_R \tilde{v}_n) = 0,
$$
where $\tilde{v}_n := v_n(\cdot - y_n)$. As $(v_n)$, $(\tilde{v}_n)$ is bounded in $BV(\mathbb{R}^N)$ and then $\tilde{v}_n \to \tilde{v}$ in $L_{loc}^1(\mathbb{R}^N)$, up to a subsequence, where $\tilde{v} \neq 0$ by Lemma \ref{lemalions}. Then, as before, we get that
$$
\Phi_\infty'(\tilde{v})\tilde{v} \leq 0
$$
and then there exists $\tilde{t} \in (0,1]$ such that $\tilde{t}\tilde{v} \in \mathcal{N}_\infty$. Hence in the same way that in Claim \ref{Claimunbounded} this will lead us to the contradiction that $c_0 = \lim_{n \to \infty}c_{\epsilon_n} \geq c_\infty$.
\end{proof}

\begin{corollary}
If $\epsilon_n \to 0$, then up to a subsequence, $\epsilon_n y_n \to y^*$ where
$$
V(y^*) = V_0 = \inf_{\mathbb{R}^N}V.
$$
\end{corollary}
\begin{proof}
If $\epsilon_n \to 0$, since by Lemma \ref{yepsilonlimitada} $(\epsilon_n y_n)_{n \in \mathbb{N}}$ is bounded, then $\epsilon_n y_n \to y^* \in \mathbb{R}^N$ up to a subsequence. As in the proof of Claim \ref{Claimunbounded} of Proposition \ref{propositionconvergence} and of Lemma \ref{yepsilonlimitada}, it is possible to prove that 
$$
c_0 = \lim_{n \to \infty}c_{\epsilon_n} \geq c_{V(y^*)} \geq c_0,
$$
where $c_{V(y^*)}$ is the mountain pass minimax level of problem (\ref{Pintrov}) with $V(y^*)$ playing the role of $V(\epsilon x)$. Then it follows that $V(y^*) = \inf_{\mathbb{R}^N}V$.
\end{proof}

\begin{lemma}
If $\epsilon_n \to 0$, then there exists $\tilde{v} \in BV(\mathbb{R}^N)$ such that
\begin{equation}
\tilde{v}_n:= v_n(\cdot - y_n) \to \tilde{v} \quad \mbox{in $L^1_{loc}(\mathbb{R}^N)$}
\end{equation}
and
\begin{equation}
f(\tilde{v}_n)\tilde{v}_n \to f(\tilde{v})\tilde{v} \quad \mbox{in $L^1(\mathbb{R}^N)$}.
\label{convergencef}
\end{equation}
\end{lemma}
\begin{proof}
First of all, note that as in Lemma \ref{lemmabounded}, it is possible to prove that $(v_n)$ is a bounded sequence in $BV(\mathbb{R}^N)$ and then that $v_n \to \tilde{v}$ in $L^q_{loc}(\mathbb{R}^N)$ for all $1 \leq q < 1^*$, where $\tilde{v} \in BV(\mathbb{R}^N)$. As in the proof of Lemma \ref{yepsilonlimitada}, it is possible to prove that $\tilde{t} \in (0,1]$ such that $\tilde{t}\tilde{v} \in \mathcal{N}_{V(y^*)} = \mathcal{N}_{V_0}$ should verify $\tilde{t} = 1$. Hence $\tilde{v} \in \mathcal{N}_0$ and note that $\Phi_0(\tilde{v}) = c_0$. In fact
\begin{eqnarray*}
c_0 & \leq & \Phi_0(\tilde{v})\\
& = & \Phi_0(\tilde{v}) - \Phi_0'(\tilde{v})\tilde{v}\\
& = & \int_{\mathbb{R}^N}\left(f(\tilde{v})\tilde{v} - F(\tilde{v})\right) dx\\
& \leq & \liminf_{n \to \infty}\int_{\mathbb{R}^N}\left(f(\tilde{v}_n)\tilde{v_n} - F(\tilde{v}_n)\right) dx\\
& = & \lim_{n \to \infty}\left(\Phi_{\epsilon_n}(v_n) - \Phi_{\epsilon_n}'(v_n)v_n\right)\\
& = & \lim_{n \to \infty} c_{\epsilon_n}\\
& = & c_0.
\end{eqnarray*}
Then 
$$
\lim_{n \to \infty}\int_{\mathbb{R}^N}\left(f(\tilde{v}_n)\tilde{v_n} - F(\tilde{v}_n)\right) dx = \int_{\mathbb{R}^N}\left(f(\tilde{v})\tilde{v} - F(\tilde{v})\right) dx
$$
and hence $f(\tilde{v}_n)\tilde{v}_n - F(\tilde{v}_n) \to f(\tilde{v})\tilde{v} - F(\tilde{v})$ in $L^1(\mathbb{R}^N)$. Thereby, by $(f_4)$, 
$$
f(v_n)v_n \to f(\tilde{v})\tilde{v} \quad \mbox{in $L^1(\mathbb{R})$.}
$$

\end{proof}

As a consequence of the last result, we can finish the proof of Theorem \ref{theoremapplication1}, by proving (\ref{concentrationf}). In fact, if $\epsilon_n \to 0$, as $n \to \infty$, denoting $\displaystyle L = \int_{\mathbb{R}^N}f(\tilde{v})\tilde{v}dx$, for a given $\delta > 0$, by (\ref{convergencef}), there exists $R > 0$ and $n_0 \in \mathbb{N}$ such that, for $n \geq n_0$,
\begin{equation}
\int_{B_R^c(0)}f(\tilde{v}_n)\tilde{v}_ndx < \delta.
\label{estimatef1}
\end{equation}
From which it follows that
\begin{equation}
\int_{B_R(0)}f(\tilde{v}_n)\tilde{v}_ndx \geq L - \delta + o_n(1).
\label{estimatef2}
\end{equation}
By the change of variable $\tilde{v}_n(x) = u_n(\epsilon_n x + \epsilon_n y_n)$, (\ref{estimatef1}) and (\ref{estimatef2}) imply that
\begin{equation}
\int_{B_{\epsilon_nR}^c(\epsilon_n y_n)}f(u_n)u_n dx < \epsilon_n^N\delta
\label{estimatef3}
\end{equation}
and
\begin{equation}
\int_{B_{\epsilon_nR}(\epsilon_n y_n)}f(u_n)u_n dx \geq C \epsilon_n^N,
\label{estimatef4}
\end{equation}
for $n \geq n_0$, where $C > 0$.
Taking into account the fact that $\epsilon_n y_n \to x_0$ where $V(x_0) = V_0$, we can consider $\overline{R} > 0$ such that, for $n \geq n_0$, $B_R(\epsilon_n y_n) \subset B_{\overline{R}}(x_0)$. Then from (\ref{estimatef3}) and (\ref{estimatef4}),
$$
\int_{B_{\epsilon_n\overline{R}}^c(x_0)}f(u_n)u_n dx < \epsilon_n^N\delta
$$
and
$$
\int_{B_{\epsilon_n\overline{R}}(x_0)}f(u_n)u_n dx \geq C \epsilon_n^N,
$$
for $n \geq n_0$, what finishes de proof of Theorem \ref{theoremapplication1}.

\section{Existence of solutions in the asymptotic constant case}

In this section we prove Theorem \ref{theoremapplication2} and then we consider the assumptions $(f_1) - (f_5)$ and $(V_1)$ and $(V_3)$. As can be seen in the statement of Theorem \ref{theoremapplication2}, our existence result is independent of $\epsilon > 0$ and then we can suppose without lack of generality that $\epsilon = 1$. Then, in this section $\|\cdot\|_1$ denotes $\|\cdot\|_\epsilon$ when $\epsilon = 1$.

Let us define $\Phi: BV(\mathbb{R}^N) \to \mathbb{R}$ by
$$
\Phi(u) = \int_{\mathbb{R}^N}|Du| + \int_{\mathbb{R}^N}V(x)|u|dx - \int_{\mathbb{R}^N}F(u)dx
$$
and consider $\Phi_\infty$ like in Section 2.

By the conditions on $f$, as in the last section, it is easy to see that $\Phi$ and $\Phi_\infty$ satisfy the geometric conditions of the Mountain Pass Theorem and then it is well defined the minimax levels 
$$
c = \inf_{\gamma \in \Gamma}\sup_{t \in [0,1]}\Phi(\gamma(t)),
$$
$$
c_\infty = \inf_{\gamma \in \Gamma_\infty}\sup_{t \in [0,1]}\Phi_\infty(\gamma(t))
$$
where 
$$
\Gamma = \{\gamma \in C([0,1],BV(\mathbb{R}^N)); \, \gamma(0) = 0 \, \mbox{and} \, \Phi(\gamma(1)) < 0\}
$$ 
and 
$$
\Gamma_\infty = \{\gamma \in C([0,1],BV(\mathbb{R}^N); \, \gamma(0) = 0 \, \mbox{and} \, \Phi_\infty(\gamma(1)) < 0\}.
$$
By $(V_3)$, it is easy to see that $\Phi(u) \leq \Phi_\infty(u)$ for all $u \in BV(\mathbb{R}^N)$ and as a consequence, 
\begin{equation}
c \leq c_\infty.
\label{ccinfty}
\end{equation}

By the results in \cite{FigueiredoPimenta1, FigueiredoPimenta2} together with the arguments explored in Section 5, it follows that there exists a critical point of $\Phi_\infty$, $w_\infty \in BV(\mathbb{R}^N)$, such that $\Phi_\infty(w_\infty) = c_\infty$. Also by \cite{FigueiredoPimenta2}, it is possible to define the Nehari manifolds associated to $\Phi$ and $\Phi_\infty$, respectively by
$$
\mathcal{N} = \{v \in BV(\mathbb{R}^N)\backslash\{0\}; \, \Phi'(v)v = 0\}
$$
and
$$
\mathcal{N}_\infty = \{v \in BV(\mathbb{R}^N)\backslash\{0\}; \, \Phi_\infty'(v)v = 0\}.
$$

By the discussion in \cite{FigueiredoPimenta1} it follows that $c = \inf_{\mathcal{N}}\Phi$ and $c_\infty = \inf_{\mathcal{N}_\infty}\Phi_\infty$. Moreover, it has been proved there that if there exists $u_0 \in BV(\mathbb{R}^N)$ such that $\Phi(u_0) =  \inf_{\mathcal{N}}\Phi$, then $u_0$ is a bounded variation solution of (\ref{Pintro}).

In order to effectively start with the proof of Theorem \ref{theoremapplication2} let us consider the two possible cases about $c$ and $c_\infty$.

\begin{itemize}
\item {\bf Case 1: $c = c_\infty$}.  If this situation occurs, problem (\ref{Pintro}) has a ground state solution. In fact, since $w_\infty \in \mathcal{N}_\infty$, then
$$
\int_{\mathbb{R}^N}|Dw_\infty| + \int_{\mathbb{R}^N}V(x)|w_\infty|dx \leq \int_{\mathbb{R}^N}|Dw_\infty| + \int_{\mathbb{R}^N}V_\infty|w_\infty|dx = \int_{\mathbb{R}^N}f(w_\infty)w_\infty dx,
$$
i.e.
$$
\Phi'(w_\infty)w_\infty \leq 0.
$$
Then there exists $t \in (0,1]$ such that $tw_\infty \in \mathcal{N}$. Hence, by $(f_5)$, 
\begin{eqnarray*}
c & \leq & \Phi(tw_\infty)\\
& = & \Phi(tw_\infty) - \Phi'(tw_\infty)tw_\infty\\
& = & \int_{\mathbb{R}^N}\left(f(tw_\infty)tw_\infty - F(t w_\infty)\right) dx\\
& \leq & \int_{\mathbb{R}^N}\left(f(w_\infty)w_\infty - F( w_\infty)\right) dx\\
& = & \Phi_\infty(w_\infty)\\
& = & c_\infty\\
& = & c.
\end{eqnarray*}
This means that $t=1$ and $w_\infty$ is also a minimizer of $\Phi$ on $\mathcal{N}$ and then is a ground-state bounded variation solution of (\ref{Pintro}).

\item {\bf Case 2: $c < c_\infty$.} By \cite[Theorem 4]{FigueiredoPimenta2}, there exist $(u_n) \subset BV(\mathbb{R}^N)$ such that 
\begin{equation}
\lim_{n \to \infty}\Phi(u_n) = c
\label{MP1}
\end{equation}
and
\begin{equation}
\|w\|_1 - \|u_n\|_1 \geq \int_{\mathbb{R}^N}f(u_n)(w - u)dx - \tau_n \|w - u_n\|_1, \quad \forall w \in BV(\mathbb{R}^N),
\label{MP2}
\end{equation}
where $\tau_n \to 0$, as $n \to \infty$.

As in Lemma \ref{lemmabounded}, it is possible to prove that $(u_n)$ is a bounded sequence in $BV(\mathbb{R}^N)$. By the compactness of the embeddings of $BV(\mathbb{R}^N)$ in $L^q_{loc}(\mathbb{R}^N)$ for $1 \leq q < 1^*$, it follows that there exists $u_0 \in BV_{loc}(\mathbb{R}^N)$ such that 
$$
u_n \to u_0 \quad \mbox{in $L^q_{loc}(\mathbb{R}^N)$ for $1 \leq q < 1^*$}
$$
and
$$
u_n \to u_0 \quad \mbox{a.e. in $\mathbb{R}^N$,}
$$
as $n \to +\infty$. Note that as in the last section, it is possible to prove that $u_0 \in BV(\mathbb{R}^N)$. Moreover, as in (\ref{equ_n0}), 
\begin{equation}
|u_n|_1 \not \to 0, \quad \mbox{as $n \to +\infty$.}
\label{u_nnot0}
\end{equation}

As in the proof of Proposition \ref{propositionconvergence}, let us use the {\it Concentration of Compactness Principle} of Lions \cite{Lions} to the following bounded sequence in $L^1(\mathbb{R}^N)$,
$$
\rho_n(x):= \frac{|u_n(x)|}{|u_n|_1}.
$$
By such a principle, one and only one of the following statements hold:
\begin{description}
\item [{\it (Vanishing)}] $\displaystyle \lim_{n \to +\infty}\sup_{y \in \mathbb{R}^N} \int_{B_R(y)}\rho_n dx = 0$, $\forall R >0$;
\item [{\it (Compactness)}] There exist $(y_n) \subset \mathbb{R}^N$ such that for all $\eta > 0$, there exists $R > 0$ such that
\begin{equation}
\int_{B_R(y_n)}\rho_n dx \geq 1 - \eta, \quad \forall n \in \mathbb{N};
\label{compactness2}
\end{equation}
\item [{\it (Dichotomy)}] There exist $(y_n) \subset \mathbb{R}^N$, $\alpha \in (0,1)$, $R_1 > 0$, $R_n \to +\infty$ such that the functions $\displaystyle \rho_{n,1}(x) := \chi_{B_{R_1}(y_n)}(x)\rho_n(x)$ and $\displaystyle \rho_{n,2}(x) := \chi_{B_{R_n}^c(y_n)}(x)\rho_n(x)$ satisfy
\begin{equation}
\int_{\mathbb{R}^N}\rho_{1,n} dx \to \alpha \quad \mbox{and} \quad \int_{\mathbb{R}^N}\rho_{2,n} dx \to 1 - \alpha.
\label{dichotomy2}
\end{equation}
\end{description}
\end{itemize}

Note that {\it Vanishing} does not occur, otherwise, by \cite{FigueiredoPimenta2}, it would hold that $\rho_n \to 0$ in $L^q(\mathbb{R}^N)$, for all $1 \leq q < 1^*$. Taking (\ref{u_nnot0}) into account, this would imply that $u_n \to 0$ in $L^q(\mathbb{R}^N)$, for all $1 \leq q < 1^*$ and then this would led us to $c = 0$, a clear contradiction.

The case in which {\it Dichotomy} takes place, we get a contradiction in both situations, when $(y_n)$ is a bounded or an unbounded sequence, just repeating the arguments in the proof of Proposition \ref{propositionconvergence}.

Then it follows that {\it Compactness} holds and then, as in the proof of Claim \ref{claimconcentration} we can prove that $(y_n)$ is a bounded sequence. Then, for $\eta > 0$, let $R > 0$ such that (\ref{compactness2}) holds and note that this implies that 
$$
\int_{B_R^c(0)}\rho_n dx < \eta, \quad \forall n \in \mathbb{N},
$$
which is equivalent to
\begin{equation}
\int_{B_R^c(0)}|u_n| dx \leq \eta|u_n|_1 \leq C\eta, \quad \forall n \in \mathbb{N}.
\label{decay12}
\end{equation}
Since $u_0 \in L^1(\mathbb{R}^N)$, there exists $R_0 > 0$ such that
\begin{equation}
\int_{B_{R_0}^c(0)}|u_0| dx \leq \eta.
\label{decay22}
\end{equation}
Then, for $R_1 \geq \max\{R,R_0\}$, since $u_n \to u_0$ in $L^1(B_{R_1}(0))$, there exists $n_0 \in \mathbb{N}$ such that
\begin{equation}
\int_{B_{R_1}(0)}|u_n - u_0| dx \leq \eta.
\label{decay32}
\end{equation}
Then, by (\ref{decay12}), (\ref{decay22}) and (\ref{decay32}), it follows that if $n \geq n_0$,
$$
\int_{\mathbb{R}^N}|u_n - u_0| dx \leq \eta + \int_{B_{R_1}^c(0)}|u_n - u_0| dx \leq \eta + \int_{B_{R_1}^c(0)}|u_n| dx + \int_{B_{R_1}^c(0)}|u_0| dx \leq C_1\eta.
$$
Hence $u_n \to u_0$ in $L^1(\mathbb{R}^N)$ and since $(u_n)$ is bounded in $L^{1^*}(\mathbb{R}^N)$, by interpolation inequality it follows that
\begin{equation}
u_n \to u_0 \quad \mbox{in $L^q(\mathbb{R}^N)$, for all $1 \leq q < 1^*$.}
\label{convergenceun}
\end{equation}

From (\ref{convergenceun}), $(f_1)$ and $(f_2)$ it follows that
\begin{equation}
\int_{\mathbb{R}^N}f(u_n)u_ndx \to \int_{\mathbb{R}^N}f(u_0)u_0 dx, \quad \mbox{as $n \to +\infty$.}
\label{convergencefun}
\end{equation}

Then from (\ref{PSepsilon}), (\ref{convergencef}) and the lower semicontinuity of $\|\cdot \|_1$ w.r.t. the $L^1(\mathbb{R}^N)$ convergence imply that
\begin{equation}
\|w\|_1 - \|u_0\|_1 \geq \int_{\mathbb{R}^N}f(u_0)(w - u_0)dx, \quad \forall w \in BV(\mathbb{R}^N),
\label{solution1}
\end{equation}
and then $u_0$ is in fact a nontrivial solution of (\ref{Pintro}). Moreover, note that from (\ref{MP1})
\begin{eqnarray*}
c & \leq & \Phi (u_0)\\
& = & \Phi (vu_0) - \Phi'(u_0)u_0\\
& = & \int_{\mathbb{R}^N}\left(f(u_0)u_0 - F(u_0)\right)dx\\
& \leq & \liminf_{n \to \infty}\int_{\mathbb{R}^N}\left(f(u_n)u_n - F(u_n)\right)dx\\
& = & \Phi(u_n) + o_n(1)\\
& = & c.
\end{eqnarray*}
Then $u_0$ is a ground-state bounded variation solution of (\ref{Pintro}) and this finish the proof of Theorem \ref{theoremapplication2}.

\section{Existence of ground state solution for autonomous case}
\label{Appendix}

In this short section, let us prove that there exists a ground-state solution to the autonomous problem
\begin{equation}
\left\{
\begin{array}{rr}
- \displaystyle \Delta_1 v + V_\infty\frac{v}{|v|} & = f(v) \quad \mbox{in $\mathbb{R}^N$,}\\
& u \in BV(\mathbb{R}^N).
\end{array} \right.
\label{Pautonomo}
\end{equation}

Let $\Phi_\infty$, $c_\infty$ and $\mathcal{N}_\infty$ defined as in Section \ref{sectionexistence}. By \cite[Theorem 1.4]{FigueiredoPimenta2}, there exists $(w_n)\subset BV(\mathbb{R}^N)$ such that $\Phi_\infty(w_n) \to c_\infty$ and moreover, 
\begin{equation}
\|w_n\|_\infty - \|w_n\|_\infty \geq \int_{\mathbb{R}^N}f(w_n)(w - w_n)dx, \quad \forall w\in BV(\mathbb{R}^N).
\label{eqsolutioninfty}
\end{equation}

As in the proof of Lemma \ref{lemmabounded}, it is possible to prove that $(w_n)$ is a bounded sequence in $BV(\mathbb{R}^N)$ and then $w_n \to w_\infty$ in $L^q_{loc}(\mathbb{R}^N)$, where $1 \leq q < 1^*$. It follows as in (\ref{vBV}) that $w_\infty \in BV(\mathbb{R}^N)$. Also, there exist $R, \beta > 0$ and a sequence $(y_n) \subset \mathbb{R}^N$ such that
\begin{equation}
\liminf_{n \to +\infty}\int_{B_R(y_n)}|w_n| dx \geq \beta.
\label{Lionscinfty}
\end{equation}
In fact, otherwise, by \cite[Theorem 1.1]{FigueiredoPimenta2}, $w_n \to 0$ in $L^q(\mathbb{R}^N)$ for $1 < q < 1^*$ and then, by $(f_2)$ and $(f_3)$, $\Phi_\infty(w_n) \to 0$, leading to a clear contradiction with the fact that $c_\infty > 0$. Then it follows that $w_\infty \neq 0$.

By proceeding exactly as in the proof of Proposition \ref{propositionconvergence}, Claim 1, it follows that $\Phi_\infty'(w_\infty)w_\infty \leq 0$. Then there exists $t_\infty \in (0,1]$ such that $t_\infty w_\infty \in \mathcal{N}_\infty$. Note also that
\begin{equation}
c_\infty + o_n(1) = \Phi_\infty(w_n) +o_n(1) = \Phi_\infty(w_n) - \Phi'_\infty(w_n)w_n = \int_{\mathbb{R}^N}\left(f(w_n)w_n - F(w_n)\right) dx.
\label{convergenciafinfty}
\end{equation}

Then applying Fatou's Lemma in the last inequality together with $(f_4)$, it follows that
\begin{eqnarray*}
c_\infty & \geq &  \int_{\mathbb{R}^N}\left(f(w_\infty)w_\infty - F(w_\infty)\right) dx\\
& \geq & \int_{\mathbb{R}^N}\left(f(t_\infty w_\infty)t_\infty w_\infty - F(t_\infty w_\infty)\right) dx\\
& = & \Phi_\infty(t_\infty w_\infty) - \Phi'_\infty(t_\infty w_\infty)t_\infty w_\infty\\
& = & \Phi_\infty(t_\infty w_\infty)\\
& \geq & c_\infty.
\end{eqnarray*}

Hence, $t_\infty = 1$, $\Phi_\infty(w_\infty) = c_\infty$ and, by \cite[Theorem 5]{FigueiredoPimenta1}, it follows that $w_\infty$ is a ground-state bounded variation solution of (\ref{Pautonomo}).

\noindent \textbf{Acknowledgment} M.T.O. Pimenta has been supported by FAPESP and CNPq/Brazil 442520/2014-0. C.O. Alves was partially supported by CNPq/Brazil  304036/2013-7  and INCT-MAT.


\begin{thebibliography}{99}

\bibitem{ABC} {\sc A. Ambrosetti, M. Badiale and S. Cingolani}, Semiclassical states of nonlinear Schr\"{o}dinger equations,  Arch. Rational Mech. Anal., 140 (1997), 285-300.


\bibitem{Anzellotti} {\sc G. Anzellotti}, {\em The Euler equation for functionals with linear growth}, Trans. Amer. Math. Soc., 290, No. 2, 483 - 501 (1985).

\bibitem{Buttazzo} {\sc H. Attouch, G. Buttazzo and G. Michaille}, {\em Variational analysis in Sobolev and BV spaces: applications to PDEs and optimization}, MPS-SIAM, Philadelphia (2006).

%\bibitem{Alexandru} {\sc Z. Balogh and A. Krist\'aly}, {\em Lions-type compactness and Rubik actions on the Heisenberg group}, Calc. Var. Partial Differential Equations, 48, No. 1-2, 89 - 109 (2013).

\bibitem{Chang} {\sc K. Chang}, {\em Variational methods for non-differentiable functionals and their applications to partial differential equations}, J. Math. Anal. Appl., 80, 102 - 129 (1981).

\bibitem{Chang2} {\sc K. Chang}, {\em The spectrum of the $1-$Laplace operator}, Commun. Contemp. Math., 9, No. 4, 515 - 543 (2009).

\bibitem{CL97}S. {\sc Cingolani and M. Lazzo}, Multiple semiclassical standing waves for a class of nonlinear
Schr\"{o}dinger equations, Topol. Methods. Nonl. Analysis 10 (1997), 1-13.

\bibitem{DF1}{\sc M. del Pino and P.L. Felmer}, Local mountain pass for semilinear elliptic problems in
unbounded domains, Calc. Var. Partial Differential Equations 4, no. 2, (1996) 121-137.


\bibitem{Clarke} {\sc F. Clarke}, {\em Generalized gradients and applications}, Trans.  Amer. Math. Soc., 205, 247 - 262 (1975).

\bibitem{DegiovanniMagrone} {\sc M. Degiovanni and P. Magrone}, {\em Linking solutions for quasilinear equations at critical growth involving the $1-$Laplace operator}, Calc. Var. Partial Differential Equations, 36, 591 - 609 (2009).

%\bibitem{Demengel} {\sc F. Demengel}, {\em Functions locally almost $1-$harmonic}, Appl. Anal., 83, No. 9, 865 - 896 (2004).

%\bibitem{Demengel1} {\sc F. Demengel}, {\em On some nonlinear partial differential equations involving the $1-$Laplacian and critical Sobolev exponent}, ESAIM Control Optim. Calc. Var., 4, 667 - 686 (1999).

\bibitem{EvansGariepy} {\sc  L. Evans and R. Gariepy}, {\em Measure theory and fine properties of functions}, CRC Press, Boca Raton, FL (1992).

\bibitem{FigueiredoPimenta1} {\sc G.M. Figueiredo and M.T.O. Pimenta}, {\em Nehari method for locally Lipschitz functionals with examples in problems in the space of bounded variation functions}, to appear.

\bibitem{FigueiredoPimenta2} {\sc G.M. Figueiredo and M.T.O. Pimenta}, {\em Strauss' and Lions' type results in $BV(\mathbb{R}^N)$ with an application to 1-Laplacian problem}, ArXiv 1610.07369v1.

\bibitem{FW}{\sc A. Floer and A. Weinstein}, Nonspreading wave packets for the cubic Schr\"{o}dinger equation with a bounded potential, J. Funct. Anal. 69, no. 3, (1986) 397-408.

\bibitem{G}{\sc C. Gui}, Existence of multi-bump solutions for nonlinear Schr\"odinger equations via variational method, Comm. Partial Differential Equations 21, no. 5-6, (1996) 787-820.

\bibitem{Giusti} {\sc  E. Giusti}, {\em Minimal surfaces and functions of bounded variation}, Birkh\"auser, Boston (1984).

\bibitem{JX} {\sc Y. Jianfu and Z. Xiping,} On the existence of nontrivial solution of quasilinear elliptic boundary value problem for unbounded domains, Acta Math. Sci. 7(3) 1987, 341-359.



\bibitem{Kawohl} {\sc B. Kawohl and F. Schuricht }, {\em Dirichlet problems for the $1-$Laplace operator, including the eigenvalue problem}, Commun. Contemp. Math., 9, No. 4, 525 - 543 (2007).

%\bibitem{Otani} {\sc J. Kobayashi and M. \^Otani }, {\em The principle of symmetric criticality for non-differentiable mappings}, J. Funct. Anal. , 214, 428 - 449 (2004).

\bibitem{Lions} {\sc  P.L. Lions}, {\em The concentration-compactness principle in the Calculus ov Variations. The Locally compact case, part 2}, Analles Inst. H. Poincar\'e Section C, 1, 223 - 283 (1984).

%\bibitem{LeonWebler} {\sc S. Leon and C. Webler}, {\em Global existence and uniqueness for the inhomogeneous $1-$Laplace evolution equation}, NoDEA Nonlinear Differential Equations Appl., 22, 1213 - 1246, (2015).

\bibitem{O2}{\sc Y.G. Oh}, On positive multi-lump bound states of nonlinear Schr\"odinger equations under multiple well potential, Comm. Math. Phys. 131, no. 2, (1990) 223-253.

\bibitem{rabinowitz}{\sc P.H. Rabinowitz}, On a class of nonlinear Schr\"odinger equations, Z. Angew. Math. Phys. 43, no. 2, (1992) 270-291.

%\bibitem{Squassina} {\sc M. Squassina}, {\em On Palais' principle for non-smooth functionals}, Nonlinear Anal., 74, 3786 - 3804 (2011).

%\bibitem{Strauss} {\sc W.A Strauss}, {\em Existence of solitary waves in higher dimensions}, Comm. Math. Phys., 55, 149 - 162 (1977).

\bibitem{Szulkin} {\sc A. Szulkin}, {\em Minimax principle for lower semicontinuous functions and applications to nonlinear boundary value problems}, Ann. Inst. H. Poincar\'e, 3, n 2,  77 - 109 (1986).

\bibitem{WX}{\sc X. Wang}, On concentration of positive bound states of nonlinear Schr\"odinger equations, Comm. Math. Phys., 153, no. 2 (1993) 229-244.

\end{thebibliography}
\end{document}